\documentclass[journal,twoside,web]{ieeecolor}
\usepackage{generic}
\usepackage{cite}
\usepackage{amsmath,amssymb,amsfonts}
\usepackage{algorithmic}
\usepackage{algorithm}
\usepackage{graphicx}
\usepackage{textcomp}
\usepackage{listings}
\usepackage{todonotes}

\newtheorem{theorem}{Theorem}
\newtheorem{remark}[theorem]{Remark}
\newtheorem{lemma}[theorem]{Lemma}

\newcommand{\red}{\color{black}}

\newcommand{\N}{\mathbb{N}}
\newcommand{\Y}{\mathbb{Y}}
\newcommand{\A}{\mathcal{A}}
\newcommand{\B}{\mathcal{B}}

\def\BibTeX{{\rm B\kern-.05em{\sc i\kern-.025em b}\kern-.08em
    T\kern-.1667em\lower.7ex\hbox{E}\kern-.125emX}}
\begin{document}
\title{Boundary Control and Estimation for Under-Balanced Drilling with Uncertain Reservoir Parameters}
\author{Timm Strecker, Ulf Jakob F. Aarsnes
\thanks{This work was supported by the Australian Research Council (LP160100666)
and the Research Council of Norway through the research center DigiWells (309589) at NORCE. \red}
\thanks{T. Strecker is with the 
Department of Electrical and Electronic Engineering, The University of Melbourne, Australia  (timm.strecker@unimelb.edu.au). }
\thanks{U. J. F. Aarsnes is with {NORCE Norwegian Research Centre AS, Oslo, Norway (ulaa@norceresearch.no)} .}}

\maketitle

\begin{abstract}
In under-balanced drilling, the bottom-hole pressure is kept below pore pressure, causing pressure dependent influx of reservoir gas into the wellbore that makes the system unstable at low drawdowns. In this paper we propose a feedback controller which stabilizes the system around an arbitrary pressure setpoint, using only topside measurement, and assuming unknown reservoir parameters. A particular challenge with this problem is the distributed and highly nonlinear nature of the system dynamics. As the control model we use the ``reduced Drift Flux Model'' which models gas-liquid flow as a nonlinear transport equation with a non-local integral source term. The observer estimates the distributed gas concentration, downhole pressure and reservoir parameters by solving the system dynamics backwards relative to how the gas rises in the well. The control inputs are then constructed by  designing target states over the next sampling period and again solving reversed dynamics to obtain the required topside pressures. The resulting controller is implemented with a 2 minute zero-order hold  to accommodate the actuation limitation situation on an actual drilling rig. The results are illustrated in simulations with a industry standard Drift Flux formulation as the plant model.
\end{abstract}


\section{Introduction}
\label{sec:introduction}

When drilling a well for the purpose of producing hydrocarbons, a slim borehole is created into a permeable pressurized formation using a drilling bit attached to a drill string. Drilling liquid is injecting into the top of the drill string and flows out through the drill bit and up the annulus around the drill string carrying formation cuttings with it out of the borehole, see Fig. \ref{fig:well schematic}.
Controlling the pressure of the drilling fluid near the bottom of the well is of key importance to the {\red success} of the drilling operation: {\red Too high pressure} means that expensive drilling liquid is lost to the formation which results in reduced return flow and insufficient hole cleaning. Too low pressure can result in pressurized formation fluids entering the well, displacing the high density drilling liquid, and creating an unstable feedback loop which can result in blow-out and collapse of the well if not controlled \cite{Godhavn2010,Godhavn2011}.

To control the pressure in the well more effectively, many wells are today drilled with a sealed anulus and a manipulated back-pressure choke, which allows for the control of the pressure at the top of the well by the driller. In particular, these tools are used to perform \textit{Under-Balanced Drilling} (UBD) where the well pressure is intentionally kept below the formation pore pressure such that formation fluids flow into the well while drilling. Underbalanced drilling have many benefits, such as improved rate of penetration, better cuttings transport, higher well productivity and less risk of loss of drilling liquid \cite{Bennion1998a}. However, these benefits come at the cost of the {\red significant} increase in the difficulty of controlling the well \cite{Graham2004}.

\begin{figure}[]
	\centering
	\includegraphics[width=\linewidth]{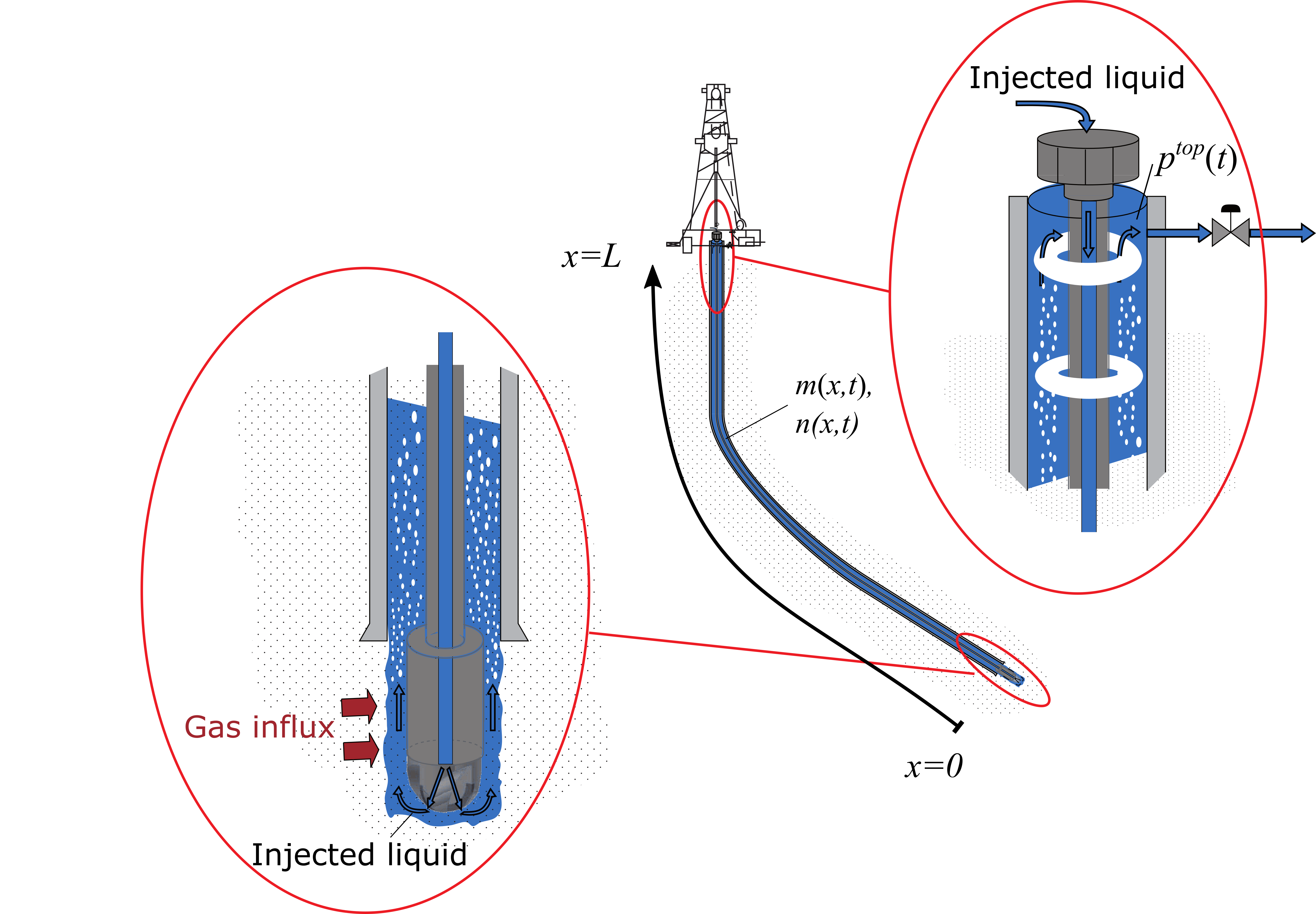}
	\caption{Schematic of an underbalanced well being drilled.}
	\label{fig:well schematic}
\end{figure}

In the context of automated pressure and flow control, the dynamics of the two-phase flow encountered in UBD is significantly more complicated than the single-phase flow of conventional drilling: In single-phase flow any operating point is inherently stable, transients are short and predictable and, barring certain well control incidents, operating conditions are reasonably homogeneous. By contrast, in two-phase underbalanced operations, the distributed gas--liquid flow and the reservoir--well interaction result in classical non-linear behavior such as multiple equilibria, limit cycles and bifurcations as described by \cite{Aarsnes2014,Mykytiw2003a,Mykytiw2004}.

A particular challenge with UBD is the interaction between the well and the reservoir, wherein a low well pressure induces reservoir influx of low density fluids which displaces the high density drilling liquid reducing the hydrostatic pressure causing yet more influx. This positive feedback loop makes the well unstable at a wide range of bottomhole pressures below the balance point, see Fig. \ref{fig:pressures}. To have stable operation in open loop, sufficient influx is required such that the frictional pressure loss caused by the influx becomes greater than the reduction in the hydrostatic pressure. Consequently, UBD is currently limited to formations with very high collapse pressure-margins. As such, there is a significant value proposal in using automated pressure control to stabilize the open loop unstable region below the balance point \cite{AarsnesThesis, aarsnes2016methodology}.

\begin{figure}
	\includegraphics[width=\linewidth]{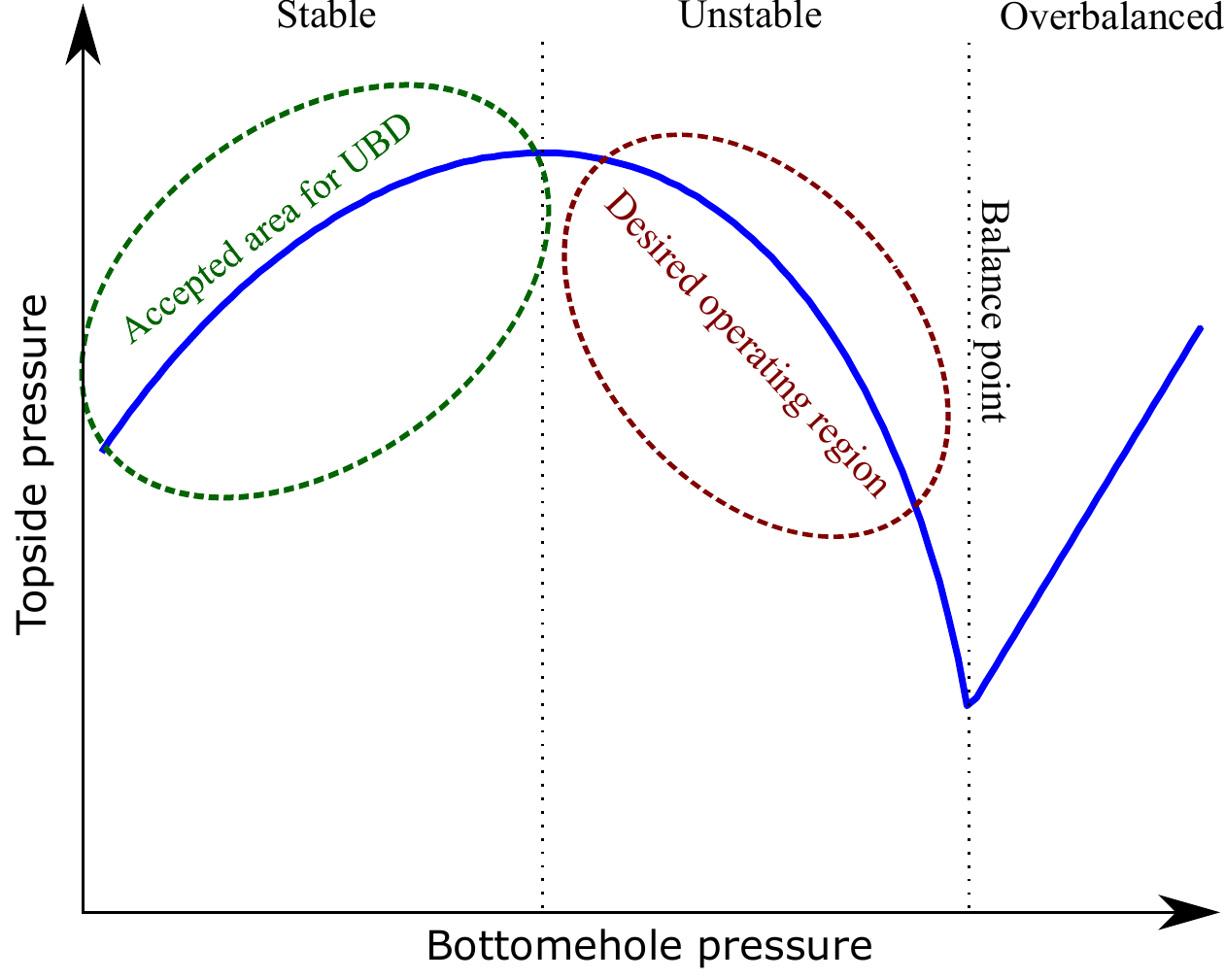}
	\caption{Conceptual plot showing the relation between topside and bottomhole pressures at equilibrium. For a given topside pressure there are typically three equilibrium points: one overbalanced, one open loop \textit{unstable} underbalanced, and one open loop stable underbalanced.}
	\label{fig:pressures}
\end{figure}

However, the non-linear and distributed nature of the system makes controller design challenging. The necessity to control the system from a stable to an open loop unstable equilibrium with vastly different dynamics (from one-phase overbalanced to two-phase underbalanced flow) precludes the application of standard ``off-the shelf'' linearized controller-designs. As such, this problem motivates the distributed non-linear controller design approach pursued in the present paper where we focus on the stabilization problem which has not been explicitly adressed before. We refer to the following papers on MPC and linear multivariable control of the extended drilling process for additional context: \cite{Pedersen2013,Pedersen2018,Pedersen2017,Pedersen2015}. {\red PDE backstepping has recently become a popular method for the feedback control of linear PDEs \cite{krstic2008backstepping,vazquez2011backstepping,aamo2013disturbance}, but the non-linearity of the model disqualifies known PDE backstepping results in this case.}


The proposed sampled-time output feedback controller consists of an observer and a feedback controller. The observer estimates the current distributed gas concentration in the well based on the history of topside measurements only, i.e., without  requiring any downhole measurements of pressure or other variables. The  feedback controller maps the estimate of the current gas concentration and the reference for the bottom hole pressure, into the topside pressures that are required to achieve the desired reference. The proposed control law is model based, and takes into account the distributed non-linear dynamics, including  non-local dependencies of the terms modelling  gas expansion which depend on the weight of the whole fluid column. In simulations with a more detailed drift-flux model, the proposed control law stabilizes the system in the desired operating region below the pore pressure, see Fig \ref{fig:pressures}. Specifically, we choose a operating point just below the balance point, which is considered the most difficult region to operate in. Moreover, the proposed  estimation scheme allows  the online identification of uncertain reservoir parameters determining the gas influx, again based solely on topside measurements. This allows the adaptation of the controller while the system is operated in closed-loop control.

The proposed observer is related to the approach in \cite{li2008observability} where, starting with the history of topside measurements, the model dynamics are first solved backwards to reconstruct the past gas and pressure distribution in the well, which is then used to estimate the current state in a second step. The feedback control part builds on ideas presented in \cite{strecker2019quasilinearfirstorder,strecker2021quasilinear2x2,li2003exact,gugat2011flow}, where one starts with the reference signal and again solves the distributed, non-linear model dynamics backwards to determine the inputs that are required to achieve  reference tracking. However, these references consider different classes of systems without the  non-local dependencies. In that sense, the theoretical contributions of this paper can be seen as  an extension of this approach to a class of quasilinear hyperbolic partial differential equations (PDEs) with non-local source terms.

%

\section{Modelling}
In this section we describe the Drift-Flux Model (DFM) that we will use to simulate the plant, and then the reduced-DFM that we use for the model based control design.

\subsection{Drift-flux model} \label{sec:drift flux}
As the plant model, to represent the two-phase gas--liquid flow and pressure dynamics, we use the drift-flux model presented in \cite{Aarsnes2014a} (see \cite{Evje2002} for numerical details). The drift-flux model is an established way to represent two-phase flow in drilling in the litterature \cite{Aarsnes2016,Udegbunam2015}. Define the mass variables
\begin{align}
m&= \alpha_L \rho_L, & n&= \alpha_G \rho_G,  \label{m n definition}
\end{align}
where for $k=L,G$ denoting liquid or gas, $\rho_k$ is the density and $\alpha_k$ is the volume fraction of the respective phase. Let $p$ be the pressure and $v_k$ be the velocity of each phase. All variables above depend on  time $t\geq 0$ and spatial position $x\in[0,L]$ along the well (in curvilinear coordinates, where $x=0$ corresponds to the well bottom and $x=L$ is at the topside choke), but the arguments $(x,t)$ behind the variables are often omitted for readability, see schematic in Fig. \ref{fig:well schematic}.
The distributed mass balances for the two phases and the momentum balance for the mixture are given by
\begin{align}
\frac{\partial m}{\partial t} &= -\frac{\partial (m v_L)}{\partial x}, \label{mass balance L} \\
\frac{\partial n}{\partial t} &= -\frac{\partial (n v_G)}{\partial x}, \label{mass balance G}\\
\frac{\partial (m v_L + n v_G) }{\partial t} &= -\frac{\partial (m v_L^2 + n v_G^2)}{\partial x} -\frac{\partial p}{\partial x} - F - G. \label{momentum balance} 
\end{align}
In (\ref{momentum balance}), the gravity term is $G=(m+n)g\cos(\phi)$ where $g$ is the gravitational acceleration and $\phi$ is the inclination from vertical. The friction term is $F=\frac{f \rho v_m |v_m|}{D}$, with friction factor $f$, mixture density $\rho=m+n$, mixture velocity $v_m = \alpha_L v_L + \alpha_G v_G$ and  hydraulic diameter $D$. 

The model is completed by the following algebraic relations. The volume fractions add up to one, i.e.,
\begin{align}
\alpha_L + \alpha_G & = 1. \label{alpha sum = 1}
\end{align}
The densities depend on the pressure as given in
\begin{align}
\rho_L &= \rho_{L,0} + \frac{p}{c_L^2}, & \rho_G &=\frac{p}{c_G^2}, \label{densities}
\end{align}
{\red where $\rho_{L,0}$ denotes the liquid density in vacuum, and $c_L,c_G$ the speed of sound in liquid and gas, respectively.}
The velocities satisfy the slip law
\begin{equation}
v_G = C_0 v_L + v_{\infty},  \label{slip law}
\end{equation}
{\red where $C_0,v_\infty$ are empirical slip parameters discussed in  \cite{gavrilyuk1996lagrangian}.}
In this model, the pressure is well-defined and can be obtained by solving (\ref{m n definition}) with (\ref{alpha sum = 1})-(\ref{densities}) for $p$.
{\red For simplicity and readability, we omitted any spatial dependence of the parameters ($f$, $c_{L}$, etc.), but all  parameters can be made dependent on $x$ (e.g., due to temperature variations along the well) without any change in the proposed approach. }

\subsection{Boundary conditions} \label{sec:boundary conditions}
In this paper we assume that the pressure applied at the topside choke is a manipulated variable determined by the driller or a control law, i.e.,
\begin{equation}
p(L,t) = p^{\text{top}}(t),  \label{p_choke}
\end{equation}
with $p^{\text{top}}$ as the manipulated variable. At the well bottom, the gas inflow depends on the difference between bottomhole pressure and the reservoir pore pressure. In this paper, we use
\begin{equation}
A n(0,t)v_G(0,t) = k_G \,\max(0,\, p_{\text{res}} - p(0,t)),  \label{inflow G}
\end{equation}
where $A$ is the cross section of the annulus and $k_G$ is the gas production index, although it is straightforward to generalize the  methods presented in this paper to other nonlinear relationships. The amount of liquid injected through the bit at the well bottom,~$W^{L,inj}$,~is determined by the rig pump
\begin{equation}
A m(0,t)v_L(0,t) = W^{L,inj}(t).  \label{inflow L}
\end{equation}

\subsection{Simplified model for control design} \label{sec: simplified}
In order to make the model more amenable for model-based control design, a simplification of the drift-flux model has been proposed in \cite{aarsnes2016simplified}. Observing that the pressure dynamics in the well are magnitudes faster than the  transport of mass, a quasi-equilibrium assumption is imposed on the momentum balance (\ref{momentum balance}), so that the number of distributed equations can be reduced to just one for the continuity of the gas volume fraction. 

Using this approach, the gas volume fraction can be approximated as
\begin{equation}
\frac{\partial \bar{\alpha}_G}{\partial t} + \bar{v}_G \frac{\partial \bar{\alpha}_G}{\partial x} = \bar{E}_G,  \label{alpha dynamics simplified}
\end{equation}
where the term
\begin{equation}
\bar{E}_G = -\frac{\bar{\alpha}_G(1-C_0 \bar{\alpha}_G)\bar{v}_G}{ \bar{p}} \frac{\partial \bar{p}}{\partial x} \label{E bar}
\end{equation}
accounts for gas expansion as the pressure decreases higher up in the well. In this simplified model, the gradient of the velocity $\bar{v}_G$ and the  pressure $\bar{p}$ are 
\begin{align}
  \frac{\partial \bar{p}}{\partial x}&= -(\bar{G}+\bar{F}),  &
  \frac{\partial \bar{v}_G}{\partial x}&= - \frac{C_0 \bar{\alpha}_G \bar{v}_G}{\bar{p}}  \frac{\partial \bar{p}}{\partial x},
\end{align}
where $\bar{F}=\frac{f\, \bar{\rho} \,\bar{v}_m\, |\bar{v}_m|}{D},$ $\bar{G}=\bar{\rho}g\cos(\phi)$, $\bar{\rho}=\bar{\alpha}_G\bar{\rho}_G+(1-\bar{\alpha}_G)\bar{\rho}_L$, and $\bar{\rho}_G$ and $\bar{\rho}_L$ are as in (\ref{densities}) but with $\bar{p}$ instead of $p$. Using this, the pressure and velocity profiles  can be obtained by starting from either the topside or the bottomhole pressure and velocity, respectively, via
\begin{align}
\bar{p}(x,t) &= \bar{p}(L,t) + \int_x^L \bar{G}(\xi,t) + \bar{F}(\xi,t) d\xi,   \label{pbar start L}\\ 
\bar{v}_G(x,t)& = \bar{v}_G(L,t) - \int_x^L  \frac{\partial \bar{v}_G(\xi,t)}{\partial \xi} d\xi,  \label{vbar start L}
\intertext{or}
\bar{p}(x,t) &=  \bar{p}(0,t) - \int_0^x \bar{G}(\xi,t) + \bar{F}(\xi,t) d\xi, \label{pbar start 0}\\
\bar{v}_G(x,t)& = \bar{v}_G(0,t) +  \int_0^x  \frac{\partial \bar{v}_G(\xi,t)}{\partial \xi} d\xi.  \label{vbar start 0}
\end{align}
That is, at each location $x\in[0,L]$ along the well, the terms $\bar{v}_G$ and $\bar{E}_G$ which determine the dynamics as given in  (\ref{alpha dynamics simplified}), can be expressed as a function of either the state in the fluid column  \emph{below} $x$, i.e., over the interval $[0,x]$,  or via the state \emph{above} $x$, i.e., over the interval $[x,L]$.

An implicit expression for the inflow boundary condition at $x=0$ can be obtained be rewriting (\ref{inflow G})-(\ref{inflow L}) as
\begin{align}
&\bar{\alpha}_G(0,t) \bar{v}_G(0,t) = \frac{k_G \,\max(0,\, p_{\text{res}} {-} \bar{p}(0,t))}{A\,\bar{\rho}_G(0,t)}, \label{inflow simplified 1}\\
&\left(1-\bar{\alpha}_G(0,t)\right)\,\frac{\bar{v}_G(0,t)-v_{\infty}}{C_0} = \frac{W^{L,inj}(t)}{A\,\bar{\rho}_L(0,t)}, \label{inflow simplified 2}
\end{align}
which can be solved for $\bar{\alpha}_G(0,t)$ and $\bar{v}_G(0,t)$. Let $\bar{\alpha}_G^{\text{inflow}}(\bar{p}(0,t))=\bar{\alpha}_G(0,t)$ be defined implicitly by the solution of (\ref{inflow simplified 1})-(\ref{inflow simplified 2}).

\section{Control design}
We present an output feedback control law consisting of an observer that estimates the distributed gas concentration along the well from  measurements at the topside boundary $x=L$ only, and a feedback control law  that computes the topside pressure so that in closed-loop the bottom pressure at $x=0$ convergences to the reference value $p_{\text{ref}}$.
The control law is sampled with sampling period $\theta$, i.e., at each time step $t_k=k\,\theta$, $k\in\N$, the control input is computed for the interval $[t_k,t_{k+1}]$. 

\begin{figure}[htbp!]\centering
\includegraphics[width=.8\columnwidth]{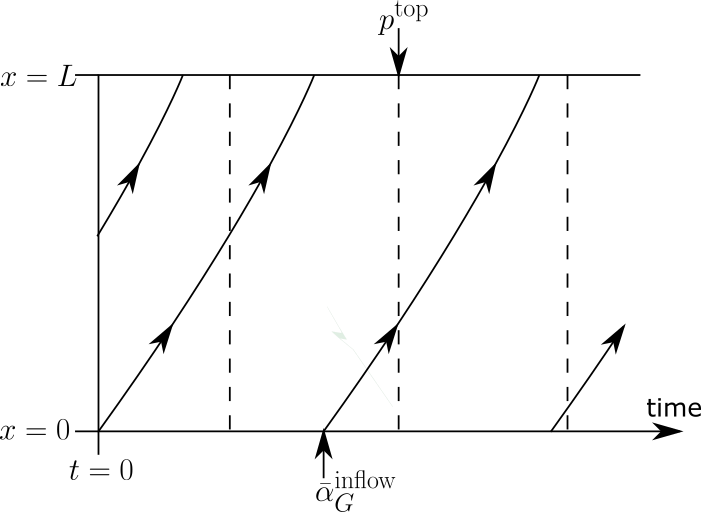}
\caption{Characteristic lines of system (\ref{alpha dynamics simplified}) representing gas propagating from the well bottom at $x=0$ to the top at $x=L$. The convex curvature of the characteristic lines indicates acceleration due to expansion. The dashed lines represent the integration paths for $\bar{p}$ and $\bar{v}_G$ as given in (\ref{pbar start L})-(\ref{vbar start 0}).  }
\label{fig:characteristic lines}
\end{figure}

The control law is based on the simplified model from Section \ref{sec: simplified}. It builds on ideas from \cite{strecker2017output,strecker2019quasilinearfirstorder,strecker2021quasilinear2x2} and is also related to \cite{li2003exact,li2008observability,gugat2011flow,li2016nodal}.  In particular, it exploits the fact that the gas propagates through the well with finite speed $\bar{v}_G$ or, mathematically speaking, along the characteristic lines of the hyperbolic PDE (\ref{alpha dynamics simplified}) (see, e.g., \cite[Chapter 2]{bressan2000hyperbolic}). The characteristic lines of system (\ref{alpha dynamics simplified}) are sketched in Figure \ref{fig:characteristic lines}.

\subsection{State estimation} \label{sec:observer}
\begin{figure}[htbp!]
\includegraphics[width=.85\columnwidth]{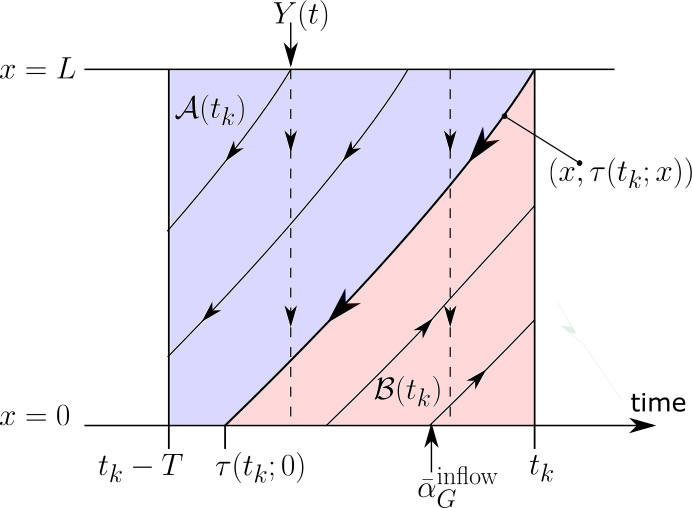}
\caption{Steps of evaluating the state estimation scheme at time $t_k$: (1) Solve the dynamics (\ref{alpha dynamics x-direction})-(\ref{v x-direction})  \emph{against} the direction of gas propagation over the domain $\A(t_k)$ (shaded in blue); (2) solve the dynamics  (\ref{alpha dynamics step2})-(\ref{v step2})  \emph{forward} in time over the domain $\B(t_k)$ (shaded in red). The thicker  line represents the characteristic line $(x,\tau(t_k;x)), x\in [0,L]$.}
\label{fig:observer steps}
\end{figure}

Due to the delay corresponding to the time the gas requires to travel from the well bottom to the top, it is impossible to immediately estimate the gas concentration along the well from topside measurements. Instead, the gas outflow at the top corresponds to gas that entered at the well bottom a certain amount of time in the past. 

Therefore, evaluating the proposed observer at  each time step $t_k$ consists of two steps, which are sketched in Figure \ref{fig:observer steps}. First, the \emph{past} gas concentration in the well is estimated by starting with the history of topside gas concentration measurements and then solving the gas dynamics (\ref{alpha dynamics simplified}) \emph{backwards} relative to how the gas propagates through the well, i.e., backwards in time and downwards in the well. Here, it is possible to reconstruct the past gas volume fraction up to the time of the characteristic line along which the latest measurement evolved (marked by the thicker line in Figure \ref{fig:observer steps}). Secondly, a prediction step is used to map the estimate of the past state on that characteristic line to the current state $\bar{\alpha}_G(\cdot,t_k)$.

Define the characteristic line corresponding to the measurement at time $t$ as
\begin{equation}
\tau(t;x) = t - \int_x^L \frac{1}{\bar{v}_G(\xi,\tau(t;\xi) )}d\xi. \label{tau}
\end{equation}
Define the measurement at time $t$ as
\begin{equation}
Y(t) =  \left(\begin{matrix}{\alpha}_G(L,t)\\{p}(L,t)\\{v}_G(L,t) \end{matrix}\right),  \label{Y(t)}
\end{equation}
and the measurement history with horizon $T>0$ as
\begin{equation}
\Y^T(t) = \left\{Y(s):\, s\in[t-T,t]\right\}.
\end{equation}  
In practice, a multi-phase flow meter can be used to measure both $\bar{\alpha}_G(L,t)$ and the topside flow rate, from which $\bar{v}_G(L,t)$ can be computed by use of (\ref{slip law}).

\subsubsection{State estimation: step 1}

With the boundary values at $x=L$ known for a sufficiently long time into the past, it is possible to estimate the past state inside the well by solving the dynamics in the negative $x$-direction. In particular, we need to assume that $T\geq t_k-\tau(t_k;0)$. By solving (\ref{alpha dynamics simplified}) for $\frac{\partial \bar{\alpha}_G}{\partial x}$ and using (\ref{pbar start L})-(\ref{vbar start L}) to determine the pressure and velocity profiles, respectively,  we obtain the following system:
\begin{align}
\frac{\partial \bar{\alpha}_G{\red (x,t)}}{\partial x} &= \frac{1}{\bar{v}_G(x,t)}\left( \bar{E}_G{\red (x,t)} -\frac{\partial \bar{\alpha}_G{\red (x,t)}}{\partial t}\right), \label{alpha dynamics x-direction} \\
\bar{\alpha}_G(L,t) &= \alpha_G(L,t), \label{BC x-direction}\\
\bar{p}(x,t) &= p(L,t) + \int_x^L \bar{G}(\xi,t) + \bar{F}(\xi,t) d\xi,   \label{p x-direction}\\ 
\bar{v}_G(x,t)& = v_G(L,t) - \int_x^L  \frac{\partial \bar{v}_G(\xi,t)}{\partial \xi} d\xi.  \label{v x-direction}
\end{align}
By use of  techniques similar to those in the proof of \cite[Theorem 5]{strecker2021quasilinear2x2} and \cite[Theorem 3.8]{bressan2000hyperbolic}, one can show that the system (\ref{alpha dynamics x-direction})-(\ref{v x-direction}) has a  solution on the determinate set
\begin{equation}
\A(t_k) = \left\{(x,t):\, x\in[0,L],\,t\in [t_k-T,\tau(t_k;x)] \right\}.  \label{A cal}
\end{equation}
Importantly, the solution on  $\A(t_k)$ contains the state on the characteristic line $(x,\tau(t_k;x))$, $x\in[0,L]$.  See also \cite[page 47]{bressan2000hyperbolic} for a more general discussion of determinate sets, and \cite[Remark 4.1]{li2008observability}  for a discussion of the minimum observation horizon $T$. In particular, the condition $T\geq t_k-\tau(t_k;0)$ ensures that the whole characteristic line $(x,\tau(t_k;x))$, for all $x\in[0,L]$, is contained in $\A(t_k)$. In other words, it ensures that the blue domain in Figure \ref{fig:observer steps} reaches the bottom boundary at $x=0$.


\subsubsection{State estimation: step 2}
The previous subsection provides a method for obtaining an estimate of the state on the characteristic line $(x,\tau(t_k;x))$. Starting with this estimate of the past state in the well, it is possible to estimate the current state by solving the following dynamics from $\tau(t_k;\cdot)$ up to current time $t_k$:
\begin{align}
 \frac{\partial \bar{\alpha}_G(x,t)}{\partial t} &+ \bar{v}_G(x,t) \frac{\partial \bar{\alpha}_G(x,t)}{\partial x}   =  \bar{E}_G(x,t),  \label{alpha dynamics step2}\\
 \bar{\alpha} _G(0,t) &= \bar{\alpha} _G^{\text{inflow}}\big(\bar{p}(0,t)\big),  \\
\bar{p}(x,t) &= \bar{p}(\tau^{\text{inv}}_k(t),t) + \int_x^{\tau^{\text{inv}_k}(t)} (\bar{G} +\bar{F})(\xi,t) d\xi,\\
\bar{v}_G(x,t) &= \bar{v}_G(\tau^{\text{inv}}_k(t),t) - \int_x^{\tau^{\text{inv}}_k(t)} \frac{\partial \bar{v}_G(\xi,t)}{\partial x} d\xi, \label{v step2}
\end{align}
where $\bar{\alpha} _G^{\text{inflow}}(\cdot)$ is defined implicitly as the solution of
 (\ref{inflow simplified 1})-(\ref{inflow simplified 2}) for a given bottom hole pressure, and for given $k\in\N$, $\tau^{\text{inv}}_k(\cdot)$ is the inverse of $\tau(t_k;\cdot)$ in the second argument, i.e., $\tau^{\text{inv}}_k(\tau(t_k;x))=x$. That is, for $t\in[\tau(t_k;0),t_k]$, $\tau^{\text{inv}}_k(t)$ gives the $x$ such that $\tau(t_k;x)=t$.

Similar to above, one can show that the system (\ref{alpha dynamics step2})-(\ref{v step2}) has a solution on the determinate set 
\begin{equation}
\B(t_k) =  \left\{(x,t):\, x\in[0,L],\,t\in [\tau(t_k;x),t_k] \right\}.  \label{B cal}
\end{equation}
Importantly, the solution on the set $\B(t_k)$ contains the estimate of the current state $\alpha(\cdot,t_k)$.


\subsubsection{State estimation: algorithm}
The preparations from the previous subsections provide the following  algorithm for estimating the state at each sampling instance $t_k$, $k\in\N$. See also Figure \ref{fig:observer steps}.
\begin{algorithm}[H]
\begin{algorithmic}[1]
\REQUIRE measurement history $\Y^T(t_k)$ for $T\geq t_k-\tau(t_k;0)$  \\
\ENSURE estimate of state $\bar{\alpha}_G(\cdot,t_k)$\\ ~

\STATE solve (\ref{alpha dynamics x-direction})-(\ref{v x-direction}) in negative $x$-direction on $\A(t_k)$, to obtain estimate of past state $\bar{\alpha}_G(x,\tau(t_k,x)$, $\bar{p}_G(x,\tau(t_k,x)$ and $\bar{v}_G(x,\tau(t_k,x)$, for all $x\in[0,L]$  \label{algo line}
\STATE  solve (\ref{alpha dynamics step2})-(\ref{v step2}) in positive $t$-direction on $\B(t_k)$, using the estimate from step \ref{algo line}. as initial condition, to obtain estimate of $\bar{\alpha}_G(\cdot,t_k)$
\end{algorithmic}
\caption{State estimation algorithm }
\label{estimation algorithm}
\end{algorithm}

For all $t_k$ satisfying $\tau(t_k;0)\geq 0$, Algorithm \ref{estimation algorithm} provides an estimate of the distributed state, including gas volume fraction and pressure, in the well. If the dynamics in the well were exactly equal to the model used for observer design, (\ref{alpha dynamics simplified})-(\ref{inflow simplified 2}), then these estimates would be equal to the actual state in the well. See also \cite{li2008observability} and \cite{strecker2017output} for related state estimation results.
 Here, the condition $\tau(t_k;0)\geq 0$ basically requires that enough time has passed since the start of measurements, that gas has had time to travel all the way from the well bottom to the top. 
 
{\red Theorems showing the well-posedness and convergence of the proposed observer when applied to the simplified drift-flux model from Section \ref{sec: simplified} are summarized in Appendix \ref{appendix}.
}

\subsection{Control law}\label{subsec:controller}
\begin{figure}[htbp!]\centering
\includegraphics[width=.7\columnwidth]{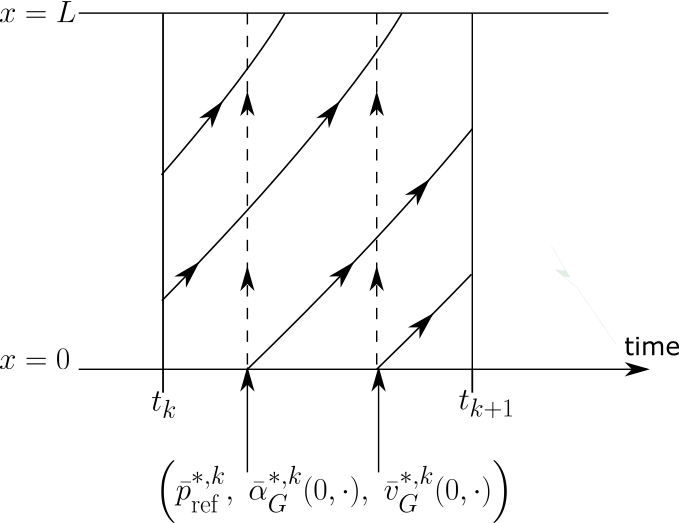}
\caption{Schematic of the computation of the control inputs over the interval $[t_k,t_{k+1}]$. Note the direction of the arrows in the dashed lines, indicating the integration path of the pressure and velocity, are in the opposite direction compared to Figure \ref{fig:observer steps}.}
\label{fig: control step}
\end{figure}

Similar to \cite{strecker2017output,strecker2021quasilinear2x2,gugat2011flow}, the idea of the control design is to start  with the desired bottom {\red pressure} values at $x=0$, which shall converge to $p_{\text{ref}}(t)$ but must also be compatible with the current state,  and to solve the {\red pressure equation} \emph{against} the propagation direction of the control input, in order to compute the trajectory that satisfies these target bottom boundary values. The control input, i.e., the topside pressure, is then set equal to the topside pressure of the target trajectory. 

For this purpose, let $\bar{\alpha}^{*,k}$, $\bar{p}^{*,k}$, etc., denote the target trajectory at the $k$-th time step, to which the system should be  equal under closed-loop control. As opposed to (\ref{p_choke}), where the control input enters at the topside boundary, we introduce a new input for the target system, $p_{\text{ref}}^{*,k}$, which is the bottom hole pressure that the target trajectory shall satisfy. The target dynamics are  given by
\begin{align}
\frac{\partial \bar{\alpha}_G^{*,k}}{\partial t} &+ \bar{v}_G^{*,k} \frac{\partial \bar{\alpha}_G^{*,k}}{\partial x} = \bar{E}_G^{*,k}, \label{target alpha}\\
\bar{p}^{*,k}(x,t) &=  {p}^{*,k}_{\text{ref}}(t) - \int_0^x \bar{G}^{*,k}(\xi,t) + \bar{F}^{*,k}(\xi,t) d\xi, \label{target p}\\
\bar{v}_G^{*,k}(x,t)& = \bar{v}^{*,k}_G(0,t) +  \int_0^x  \frac{\partial \bar{v}_G{*,k}(\xi,t)}{\partial \xi} d\xi, \label{target v}
\end{align}
where all terms $\bar{E}_G^{*,k}$, $\bar{F}^{*,k}$, etc., are defined as in Section \ref{sec: simplified} but evaluated at the target state, and
with the boundary conditions given implicitly by
\begin{align}
&\bar{\alpha}_G^{*,k}(0,t) \bar{v}_G^{*,k}(0,t) = \frac{k_G \,\max(0,\, p_{\text{res}} - {p}^{*,k}_{\text{ref}}(t))}{A\,\bar{\rho}_G^{*,k}(0,t)}, \label{target BC 1} \\
&\Big(1-\bar{\alpha}_G^{*,k}(0,t)\Big)\,\frac{\bar{v}_G^{*,k}(0,t)-v_{\infty}}{C_0} = \frac{W^{L,inj}(t)}{A\,\bar{\rho}_L^{*,k}(0,t)},\label{target BC 2}
\end{align}
and initial condition 
\begin{equation}
\bar{\alpha}_G^{*,k}(\cdot,t_k) = \bar{\alpha}_G(\cdot,t_k). \label{target IC}
\end{equation}
Note that the target bottom hole pressure, $p_{\text{ref}}^{*,k}$, enters in both (\ref{target p}) and (\ref{target BC 1}). 

The design of $p_{\text{ref}}^{*,k}$ must ensure continuity with the state estimate at time $t_k$ and should converge to the actual reference, $p_{\text{ref}}$, in a continuous fashion. Moreover, the time derivative of   $p_{\text{ref}}^{*,k}$ should remain sufficiently slow to avoid shock waves in the well or, mathematically speaking, a collision of characteristic lines. One design that satisfies these conditions is {\red
\begin{equation} 
p_{\text{ref}}^{*,k}(t)= \begin{cases} \bar{p}^k +  \bar{p}^{\prime}\cdot(t-t_k)\cdot  \operatorname{sign}(e^k) &  t \leq t_k+\frac{|e^k|}{\bar{p}^{\prime}}\\
 p_{\text{ref}} &  t > t_k+\frac{|e^k|}{\bar{p}^{\prime}}
 \end{cases}, \label{pref*}
\end{equation} 
where $\bar{p}^k = \bar{p}(0,t_k)$ and $e^k = p_{\text{ref}}-\bar{p}^k$} are  the estimated bottom hole pressure and tracking error at time $t_k$, respectively,  and $\bar{p}^{\prime}>0$ is the desired time-derivative of the bottom hole pressure. That is, $p_{\text{ref}}^{*,k}$ converges linearly with rate $\bar{p}^{\prime}$ to $p_{\text{ref}}$ and stays there once the reference is reached.
If $\bar{p}^{\prime}$ is chosen sufficiently small, one can again show that the system (\ref{target alpha})-(\ref{pref*}) is well-posed, i.e., it has a unique solution on the domain $[0,L]\times[t_k,t_{k+1}]$.

By setting 
\begin{equation}
p^{\text{top}}(t)=\bar{p}^{*,k}({\red L,}t)  \label{p_top closed loop}
\end{equation}
 for $t\in[t_k,t_{k+1}]$, and assuming exact model knowledge, the closed-loop trajectory of (\ref{alpha dynamics simplified})-(\ref{inflow simplified 2}) is equal to the target trajectory on the domain $[0,L]\times[t_k,t_{k+1}]$. In particular, the closed loop trajectory satisfies $\bar{p}(0,t)=p_{\text{ref}}^{*,k}(t)$ for all $t\in[t_k,t_{k+1}]$. See also \cite{strecker2017output,strecker2021quasilinear2x2,gugat2011flow} for comparison.
 
 The steps required to evaluate the  control law at each time step are summarized in the following Algorithm and also in Figure \ref{fig: control step}.

\begin{algorithm}[H]
\begin{algorithmic}[1]
\REQUIRE estimate of $\bar{\alpha}_G(\cdot,t_k)$ and $\bar{p}(0,t_k)$  \\
\ENSURE control input $p^{\text{top}}(t)$ for $t\in[t_k,t_{k+1}]$\\ ~

\STATE set $p^{*,k}_{\text{ref}}(t)$, $t\in[t_k,t_{k+1}]$, as per (\ref{pref*})
\STATE solve (\ref{target alpha})-(\ref{target IC}) over domain $[0,L]\times [t_k,t_{k+1}]$
\STATE set $p^{\text{top}}(t)$, $t\in[t_k,t_{k+1}]$, as per (\ref{p_top closed loop})
\end{algorithmic}
\caption{Control algorithm }
\label{control algorithm}
\end{algorithm}

{\red Well-posedness of the control law and convergence of the closed loop system when applied to the simplified drift-flux model from Section \ref{sec: simplified} are  discussed in Appendix \ref{appendix}.
}

 \subsection{Estimation of reservoir parameters} \label{sec: estimation parameters}
 
Step \ref{algo line} in Algorithm \ref{estimation algorithm} can also be used to estimate the production index $k_G$ and pressure $p_{\text{res}}$ of the reservoir, which might be uncertain in practice, using only  measurements at the topside boundary. For this, note that the system  (\ref{alpha dynamics x-direction})-(\ref{v x-direction}) does not depend on the  boundary condition at $x=0$ modelling the gas influx as given by (\ref{inflow simplified 1}). That is, Step \ref{algo line} in Algorithm \ref{estimation algorithm} provides estimates of the bottom hole gas volume fraction and pressure, $\bar{\alpha}_G(0,t)$ and $p(0,t)$, and thus the gas influx, over the past interval $t\in[t_k-T,\tau(t_k;0)]$, using only the history of topside measurements but not the boundary condition at the bottom of the well. In Figure \ref{fig:observer steps}, the time interval $t\in[t_k-T,\tau(t_k;0)]$ corresponds to the times where the blue domain reaches the bottom boundary at $x=0$. The uncertain values of $k_G$ and/or $p_{\text{res}}$ can then be estimated via  standard  least-square curve fitting.
 
For $k\in \N$, let $\theta_k^i\in[t_k-T,\tau(t_k;0)]$, $i=1,\ldots,I_k$, be sampling instances over the interval $[t_k-T,\tau(t_k;0)]$.  Let 
\begin{align}
\hat{w}_{G,k}^i &=  \bar{\alpha}_G(0,\theta_k^i)\,  \bar{v}_G(0,\theta_k^i) \, \bar{\rho}_G(0,\theta_k^i) \,A, \\
\hat{p}_k^i &=  \bar{p}(0,\theta_k^i),
\end{align} 
be the estimates of the gas influx and bottom hole pressure at these sampling instances as returned by step \ref{algo line} in Algorithm \ref{estimation algorithm}. At each time step $t_k$, the  past estimates from all previous steps up to that time can be concatenated as
\begin{align}
  \left(\begin{matrix}
\hat{W}_{G,k}^1\\ \vdots \\ \hat{W}_{G,k}^{N_k}
\end{matrix}  \right) &=
 \left(\begin{matrix}
\hat{w}_{G,1}^1\\ \vdots \\ \hat{w}_{G,1}^{I_1} \\ \hat{w}_{G,2}^1\\ \vdots \\\hat{w}_{G,2}^{I_2}\\ \vdots \\ \hat{w}_{G,k}^1\\ \vdots \\\hat{w}_{G,k}^{I_k}
\end{matrix}  \right), &
 \left(\begin{matrix}
\hat{P}_{k}^1\\ \vdots \\ \hat{P}_{k}^{N_k}
\end{matrix}  \right) &=
 \left(\begin{matrix}
\hat{p}_{1}^1\\ \vdots \\ \hat{p}_{1}^{I_1} \\ \hat{p}_{2}^1\\ \vdots \\ \hat{p}_{2}^{I_2}\\ \vdots \\ \hat{p}_{k}^1\\ \vdots \\\hat{p}_{k}^{I_k}
\end{matrix}  \right),
\end{align} 
%
 where $N_k=\sum_{j=1}^k I_j$.  By choosing $T$ sufficiently large, it can be ensured that there is no gap between sampling points at consecutive steps, $\theta_k^{I_k}$ and $\theta_{k+1}^1$. Moreover, if required the samples can be processed further to, e.g., remove duplicate samples or to ensure equal spacing. 
 
Once the past estimates of the gas influx and bottom hole pressure have been collected, the production index and reservoir pressure can be estimated by solving the optimization problem \begin{equation}
\left\{\hat{k}_G^k,\,\hat{p}^k_{\text{res}}\right\} = \operatorname*{arg\,min}_{\hat{k}_G,\,\hat{p}_{\text{res}}} \sum_{i=1}^{N_k} \left| \hat{W}_{G,k}^i - \hat{k}_G \max (0,\hat{p}_{\text{res}}-\hat{P}_{k}^i  )   \right|^2.  \label{curve fit}
\end{equation}
It should be noted that the curve fitting procedure does not have to be of the exactly of the form (\ref{curve fit}). For instance,  weights could be put on the  different samples. It is also possible to use other nonlinear functions to model the relationship between bottom hole pressure and gas influx such as polynomials of the pressure difference $\hat{p}_{\text{res}}-\hat{P}_{k}^i$.

\subsection{\red Adding an integral term } \label{sec:integral term}
{\red
The feedback control law in Section \ref{subsec:controller} can be seen as a static nonlinear feedback gain, similar to the proportional gain in classic linear control. In the presence of modelling errors, using such a static gain can lead to a tracking error at steady state (see also the simulations in Section \ref{sec:MonteCarlo}). Such steady state tracking errors can be corrected if (infrequent) measurements of the bottom hole pressure are available, by adding an correction term involving the integral of the tracking error. One of the main advantages of the controller from Section \ref{subsec:controller} is that it, combined with the observer from Section \ref{sec:observer}, only requires topside measurements. However, some, potentially infrequent downhole pressure measurements might be available in practice, in which case it is desirable to reduce any pressure tracking errors.

Let $\tilde{t}_i$ be the sampling instances where downhole pressure measurements are available, with in general slower sampling rate $\tilde{t}_{i+1}-\tilde{t}_i \gg {t}_{k+1}-{t}_k$. Define the integral term as
\begin{align}
\Pi_0 &=0, \\
 \Pi_{i}&= \Pi_{i-1} + K_I\times \left( p_{\text{ref}} - p(0,\tilde{t}_i)\right)  \times\left(\tilde{t}_i - \tilde{t}_{i-1} \right),
\end{align}
with integral gain $K_I$. Then, the topside pressure as given in (\ref{p_top closed loop}) can be modified to 
\begin{equation}
p^{\text{top}}(t)=\bar{p}^{*,k}({\red L,}t) + \Pi_{i}   \label{p_top closed loop integral term}
\end{equation}
for all $k$ with $t_k\in[\tilde{t}_i, \tilde{t}_{i+1}]$. In order to avoid that the more aggressive static feedback term compensates the much slower integral term, the pressure offset must be considered in the measurement as in
\begin{equation}
Y(t) =  \left(\begin{matrix}{\alpha}_G(L,t)\\{p}(L,t) - \Pi_i  \\{v}_G(L,t) \end{matrix}\right),  \label{Y(t) integral}
\end{equation}
with $i$ such that $t\in [\tilde{t}_i, \tilde{t}_{i+1}]$.

It should be noted that with infrequent sampling of the downhole pressure (say, in the order of once per hour), the integral term does hardly contribute to stabilization of the pressure (which would require more frequent sampling  \cite{Pedersen2018}), but only acts to reduce the steady state tracking error.

}

\section{Numerical simulation}


\begin{table}[!htbp]
\caption{Parameters}
\label{table}
\setlength{\tabcolsep}{3pt}
\begin{tabular}{| r  l || r l || r l |}
\hline
$L=$& $2500$\,m & 
$\rho_{L,0}=$&   $975$\,kg/m$^3$ &
$p_{\text{res}}=$ &  $266$\,bar   \\
$A=$&   $0.012$\,m$^2$&
$c_L=$&   $1000$\,m/s &
$k_G=$ &  $0.01$\,kg/(s\,bar)  \\
$D=$ &  $0.0635$\,m&
$c_G=$ &  $315$\,m/s&
$W^{L,inj}=$ &  $13$\,kg/s\\
$f=$&   $0.03$ & 
$C_0=$ &  $1.1$ &
$v_{\infty}=$ &  $0.1$\,m/s\\
$\theta =$ & $10$\,min &
$\bar{p}^{\prime}=$ & $10$\,bar/h&
$\alpha_G(\cdot,0)\equiv$ & $ 0$ \\
\hline
\end{tabular}
\label{table parameters}
\end{table}

\subsection{Simulation parameters}

We demonstrate the performance of the proposed control law in  numerical simulations of a well with the parameters given in Table \ref{table parameters}. The dynamics in the well are modelled using the drift-flux model introduced in Sections \ref{sec:drift flux} \ref{sec:boundary conditions} while the simplified model from Section \ref{sec: simplified} is only used for the output feedback control law. A first-order finite difference scheme with 50 discretization elements is applied to convert all PDEs (the system dynamics (\ref{mass balance L})-(\ref{momentum balance}) and all PDEs in Algorithms \ref{estimation algorithm} and \ref{control algorithm}) into high-order ODEs (``method of lines''). The resulting ODEs are then solved in matlab by use of  \lstinline{ode23tb}. 

At each sampling instance, the topside pressures are precomputed over a $\theta=10$\,minute interval. The algorithm provides a continuously varying signal for the topside pressure. However, in practice, the choke on an actual rig is usually not manipulated continuously. In  order to emulate this, a further zero-order hold with period 2\,minutes is applied to the original topside pressure signal, {\red  so that the topside pressure becomes a piecewise-constant signal that changes every 2 minutes, and attains 5 different values   over each 10\,minute period. That is, due to the zero-order hold, the actual topside pressure that is applied to the system deviates slightly from the output of Algorithm \ref{control algorithm}. }
Pre-computing the control inputs for each 10-minute period, which involves solving the PDEs outlined in Algorithms \ref{estimation algorithm} and \ref{control algorithm},  takes less than 1\,second on a standard laptop, i.e., a fraction of the sampling interval.

{\red The simulation presented below deviate from the formal analysis in Appendix \ref{appendix}, which  focuses on the simplified drift-flux model from Section \ref{sec: simplified} in closed loop with the proposed estimation and control scheme for nominal parameters, in that the simulation model is different to the model used for control design, that the control inputs are applied in a zero-order hold fashion, that the parameter identification scheme and integral action from Sections \ref{sec: estimation parameters} and \ref{sec:integral term} are applied (which were not part of the nominal design analysed in the appendix), and that uncertainty in parameters and disturbances/noise affecting the  measurements and control input are included.
Thereby, the simulations serve to demonstrate that the proposed estimation and control method not only works in the ideal case, as proven in the appendix, but also shows robustness with respect to issues that need to be expected in practical applications.
}

\subsection{Simulation results - nominal design}
\begin{figure}[htbp!]\centering
\includegraphics[width=.9\columnwidth]{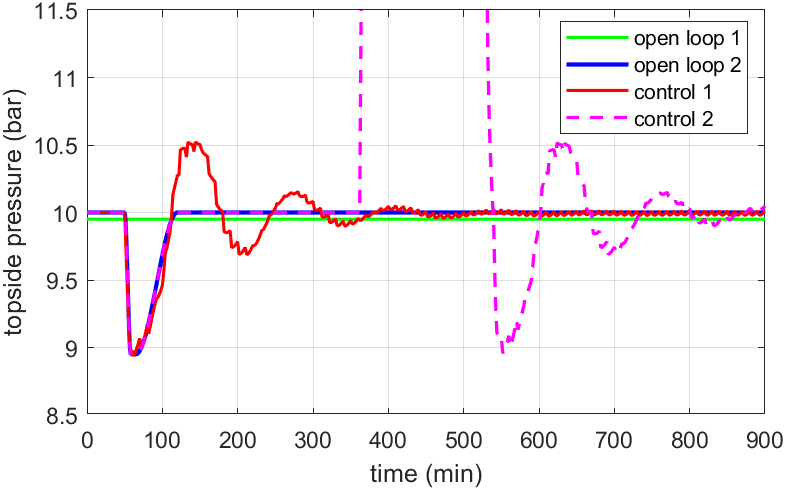}
\includegraphics[width=.9\columnwidth]{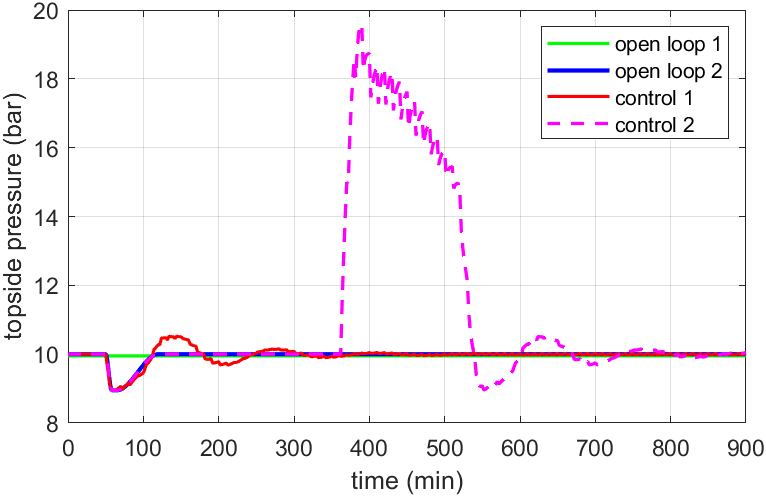}
\caption{Topside pressures as computed by  the feedback control law presented in this paper in the two scenarios, and  two open-loop alternatives. Although the topside pressures convergence to the same point, they correspond to three different equilibria, see Fig. \ref{fig:pressures}. The top figure is a zoom in of the bottom figure.}
\label{fig:ptop}
\end{figure}

\begin{figure}[htbp!]\centering
\includegraphics[width=.9\columnwidth]{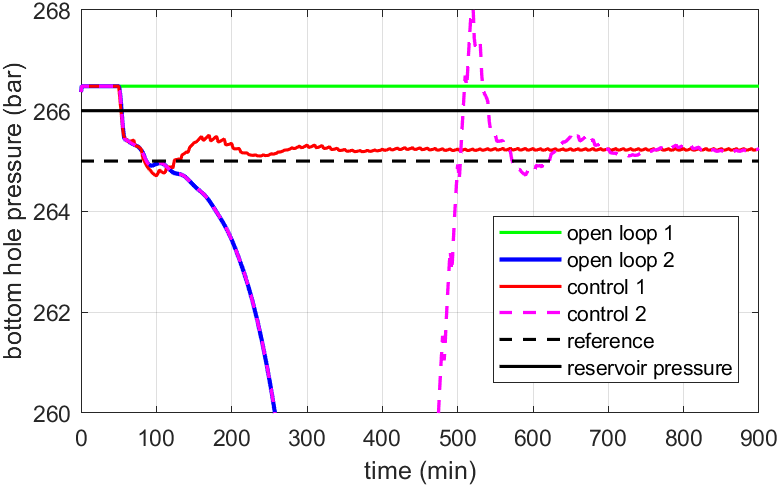}\\
\includegraphics[width=.9\columnwidth]{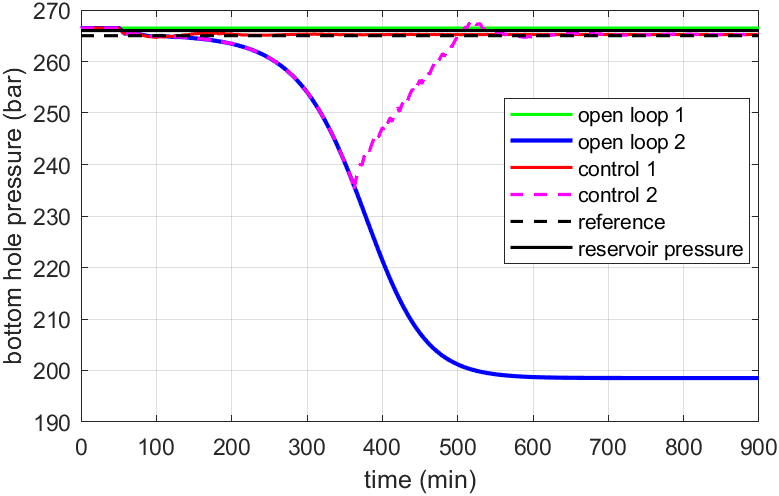}
\caption{Comparison of bottomhole pressure trajectories when using the proposed feedback control law in the two scenarios  and the two open-loop alternatives. The top figure is a zoom in of the bottom figure.}
\label{fig:pbottom}
\end{figure}

\begin{figure}[htbp!]\centering
\includegraphics[width=.9\columnwidth]{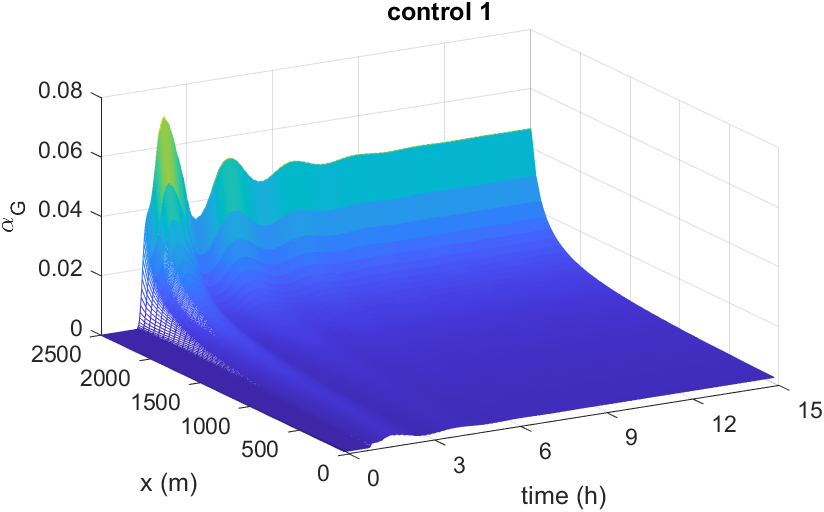}\\
\includegraphics[width=.9\columnwidth]{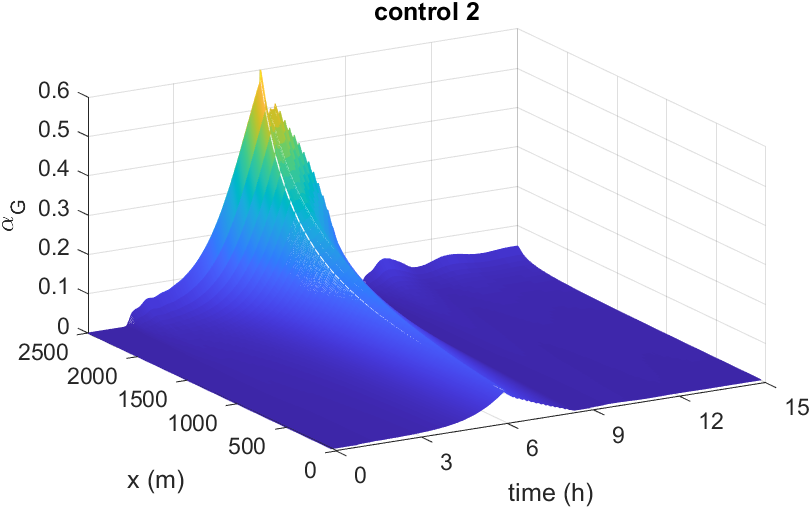}\\
\includegraphics[width=.9\columnwidth]{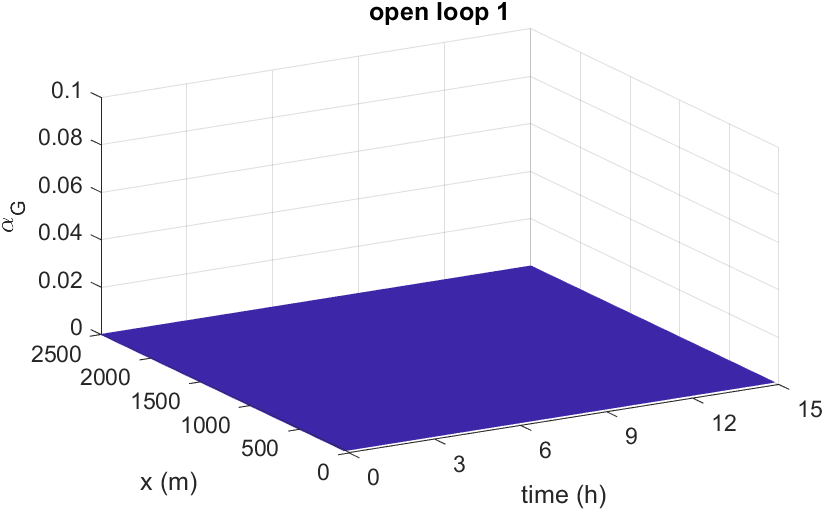}\\
\includegraphics[width=.9\columnwidth]{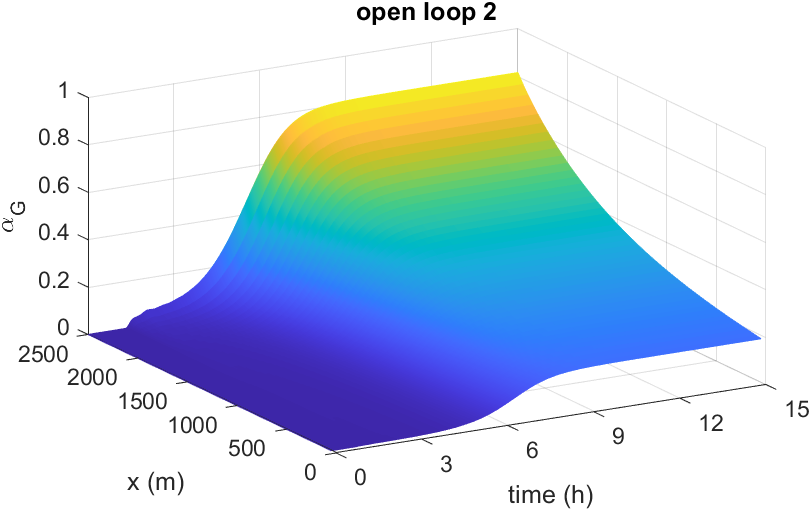}
\caption{Gas concentration $\alpha_G$ using the feedback control law presented in this paper and the two open-loop alternatives.}
\label{fig:alpha}
\end{figure}

In this section we demonstrate the controller performance in simulations where the well parameters are assumed known. The topside and bottom hole pressure trajectories are shown in Figure \ref{fig:ptop} and \ref{fig:pbottom}, respectively. The gas volume fraction is shown in Figure \ref{fig:alpha}. At the initial condition there is no gas in the well. 

In the trajectory titled ``control 1'', the topside pressure is initially held at 10\,bar until the control law is activated at $t=50$\,minutes. While the topside pressure is at 10\,bar, the bottom hole pressure sits slightly above the reservoir pressure at 266.5\,bar, so that there is no inflow of gas.  Once the controller is activated, it lowers the topside and, thus, the bottom hole pressure. Consequently, gas starts to enter the well. The presence of gas in the well lowers the pressure difference between topside and well bottom (because the light gas reduces the weight of the liquid/gas column), which further lowers the bottom hole pressure. The controller uses the estimate of $\bar{\alpha}_G$ to compensate this effect  and stabilizes the bottom hole pressure close to the reference value at 1\,bar below the reservoir pressure. Note again that the feedback controller uses no measurements of the downhole pressure, which leads to the small offset between down hole pressure reference and asymptotically achieved down hole pressure.    As shown in Figure \ref{fig:alpha}, once the bottom hole pressure is settled at 265\,bar, the gas concentration stabilizes at around 0.3\% (by area) at the well bottom and expands to approximately 6.5\% at the top of the well. The controller achieves stabilization of the bottom pressure close to the reference despite the mismatch between the drift-flux model used for simulation and the simplified model used for computation of the control.

In the second closed-loop simulation  (``control 2''), the feedback controller is only activated at time $t=6$\,hours. Before that, the topside pressure initially decreases as the trajectory in ``control 1'' in order to initiate a gas inflow, but is then eventually brought back to the equilibrium at $10$\,bar. This blow-out scenario is described in more detail in the following subsection under ``open loop 2''. Briefly speaking, the equilibrium corresponding to the reference down hole pressure is unstable, and the gas entering the well leads to a severe drop in the down hole pressure (approximately 30\,bar by the time the feedback controller is activated, leading to a gas concentration of 60\% at the top of the well). However, the feedback controller again manages to estimate the gas distribution in the well with sufficient accuracy,  compensates its effect on the pressure in the well by increasing the topside pressure for a period of time, and brings the down hole pressure back to the reference.

In these simulations, the delay $t_k-\tau(t_k,0)$ is just over 30\,minutes at all time steps. Thus, saving the measurements over a horizon of $T=40$\,minutes is a conservative choice to ensure that the steps in Algorithm \ref{estimation algorithm} are well-posed.

\subsection{Comparison with open-loop control}
For comparison, Figures \ref{fig:ptop}-\ref{fig:alpha} also show the trajectories corresponding to two open-loop topside pressure signals. In the first alternative (``open loop 1''), the topside pressure is held constantly at 10\,bar. Since there is no gas in the well at the initial condition, the bottom hole pressure remains slightly above the reservoir pressure, i.e., in an over-balanced situation, for the entire simulation.

In the second open-loop alternative (``open loop 2''), the topside pressure signal   drops like in the closed-loop simulation in order to initiate a gas inflow at the well bottom, before recovering to the average topside pressure of the closed-loop case at  10\,bar. However, the equilibrium corresponding to the reference bottom hole  pressure is unstable. That is, the gas inflow reduces the gravitational pressure drop in the well, which further reduces the bottom hole pressure and increases the gas inflow, until it reaches a  stable equilibrium at almost 70\,bar below the reference.  As shown in Figure \ref{fig:alpha}, the large gas influx due to the low bottom hole pressure leads to a gas concentration of approximately 18\% by area at the well bottom and just over 80\% at the top.

\subsection{Estimation of  reservoir parameters}

\begin{figure}[htbp!]\centering
\includegraphics[width=.9\columnwidth]{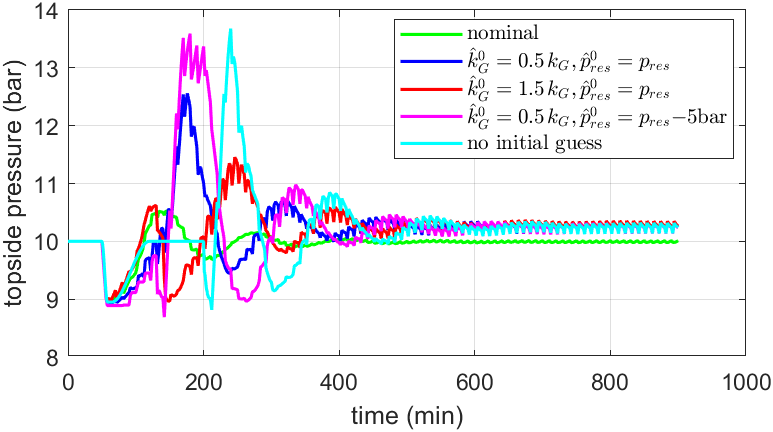}\\
\includegraphics[width=.9\columnwidth]{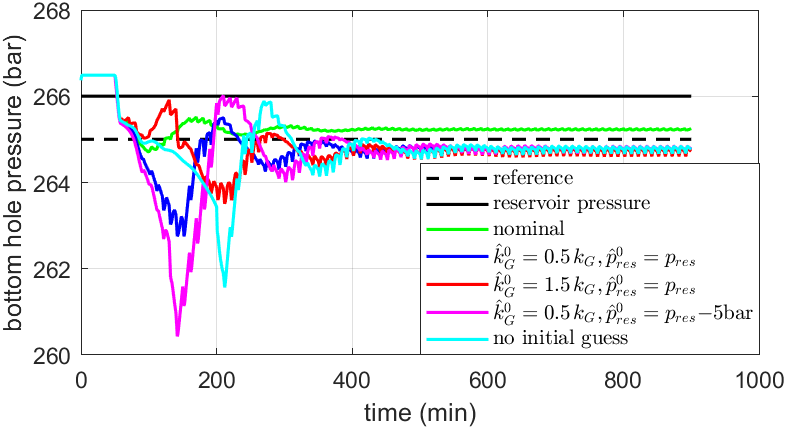}
\caption{Comparison of  pressure trajectories for nominal and uncertain reservoir parameters.}
\label{fig:pressures uncertain}
\end{figure}

In this section we demonstrate the performance of both the  controller and the  parameter estimation scheme from Section \ref{sec: estimation parameters}. Here, the  parameters from Table \ref{table parameters} are used to simulate the well but the reservoir parameters $k_G$ and $p_{\text{res}}$ are assumed to be uncertain. 

Figure \ref{fig:pressures uncertain} shows the topside and bottom hole pressure trajectories for five different simulations. In three of these simulations, the initial guess $\hat{k}_G^0$ overestimates or underestimates the actual production index by 50\%, respectively, and in one of them  the initial guess $\hat{p}_{\text{res}}^0$ of the reservoir pressure also underestimates the actual value by 5\,bar. In each of these simulations, the output feedback controller uses the initial values $\hat{k}_G^0$ and $\hat{p}_{\text{res}}^0$ until there is one instance at which the estimated gas influx $\hat{W}_{G,k}^i$ exceeds 1\,kg/min. Once this threshold is exceeded, the reservoir parameters are estimated at each following time step as described in Section \ref{sec: estimation parameters} based on the current set of samples, and the  updated estimates $\hat{k}_G^k$ and $\hat{p}_{\text{res}}^k$ are used  both Algorithm \ref{estimation algorithm} for  state estimation and in Algorithm \ref{control algorithm} to compute the control inputs. 

The only modification compared to Section \ref{sec: estimation parameters} is the inclusion of a simple data processing step, in that at each $t_k$, the new estimation samples $\hat{w}_{G,k}^i$ and $\hat{p}_k^i$, $i=1,\ldots,I_k$, are only added to the overall set of samples used for curve-fitting  if they satisfy the following condition:
\begin{align}
\min_{j\leq k-1} & \left|\operatorname*{mean}_{i=1\ldots I_k} (\hat{w}_{G,k}^i) - \operatorname*{mean}_{i=1\ldots I_j} (\hat{w}_{G,j}^i) \right| \geq 0.05\,\text{kg/min} \label{condition samples 1} \\ \intertext{or}
\min_{j\leq k-1} & \left|\operatorname*{mean}_{i=1\ldots I_k} (\hat{p}_{k}^i) - \operatorname*{mean}_{i=1\ldots I_j} (\hat{p}_{j}^i) \right| \geq 0.05\,\text{bar}.\label{condition samples 2}
\end{align}
This is to prevent that once the system settles around steady state, more and more almost identical samples  of $\hat{w}_{G,k}^i$ and $\hat{p}_{k}^i$ keep getting added. Otherwise,   excessive weight  would be put on the accumulation of samples around the equilibrium, which could ultimately cause the solution of   optimization problem (\ref{curve fit}) to slowly drift as more and more samples around steady state keep getting added.

 In another simulation scenario, no initial guess of $\hat{k}_G^0$ and $\hat{p}_{\text{res}}^0$ is used. Instead, the topside pressure is set equal to the one used in the simulation ``open loop 2'' described in the previous section. As discussed above, this open-loop signal induces a gas influx and drop in the bottom hole pressure. Once the estimated gas influx $\hat{W}_{G,k}^i$ exceeds the threshold of 1\,kg/min at one sampling instance, the estimates $\hat{k}_G^k$ and $\hat{p}_{\text{res}}^k$ are again obtained as in the previous case and the feedback controller is activated. 
For comparison, the nominal simulation where the exact values of $k_G$ and $p_{\text{res}}$  are available to the controller is also shown.

As shown in Figure \ref{fig:pressures uncertain}, the adaptive feedback controller manages to stabilize the bottom hole pressure close to the reference in all simulations. Uncertainty in the reservoir parameters does affect the solution  during transients, before the uncertain parameters are identified. In each of the four cases with uncertainty, the estimate $\hat{k}_G^k$ settles between 7-11\% below the actual value $k_G$, and $\hat{p}_{\text{res}}^k$ settles between 0.2-0.3\,bar below $p_{\text{res}}$. This error leads to the slightly lower  bottom hole pressures to which the four trajectories with uncertainty converge in Figure \ref{fig:pressures uncertain}.

\begin{figure}[htbp!]\centering
\includegraphics[width=.9\columnwidth]{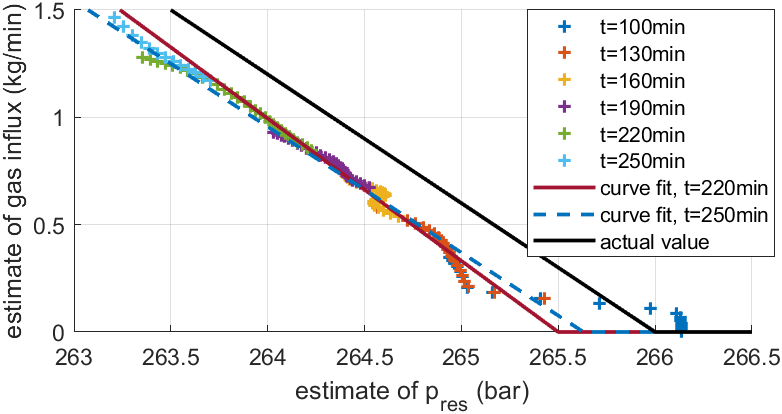}\\
\includegraphics[width=.9\columnwidth]{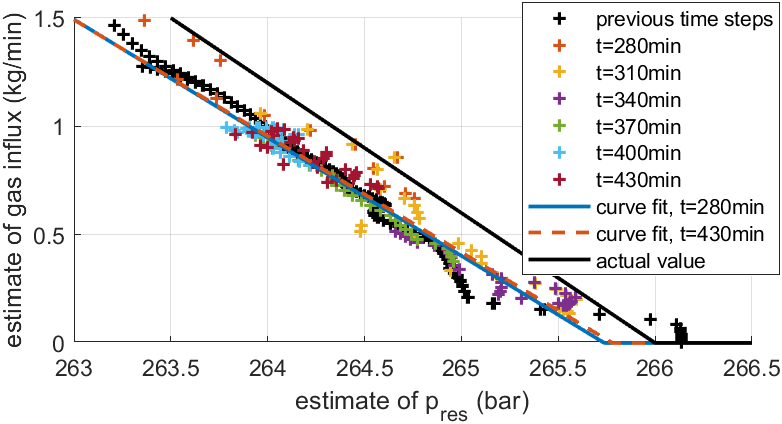}
\caption{ Samples of gas influx estimates $\hat{w}_{G,k}^i$ and bottom hole pressure estimates $\hat{p}_{k}^i$, $i=1,\ldots,I_k$, at different time steps $t_k$ (shown only for every third $t_k$), for the simulation where no initial guesses $\hat{k}_{G}^0$ and $\hat{p}_{\text{res}}^0$ are used. The lines show the  curve fit at different times and the actual relationship as given by (\ref{inflow G}).  }
\label{fig:estimate bottom BC uncertain}
\end{figure}

The parameter identification steps are investigated more closely in Figure \ref{fig:estimate bottom BC uncertain}, at the example of the trajectory where the system is initially operated in open-loop  and no initial guesses $\hat{k}_{G}^0$ and $\hat{p}_{\text{res}}^0$ are used.  The top figure shows samples of $\hat{w}_{G,k}^i$ and $\hat{p}_{k}^i$ up to time $t_k=250$\,min where the bottom hole pressure is still decreasing with time. Despite numerical errors, the samples lie close to a line (except for the very first estimated influxes up to around 0.1\,kg/min), and a good curve fit is possible once the threshold of 1\,kg/min is exceeded. When the pressure increases (corresponding to the samples shown in the bottom figure), the influx estimates tend to be slightly higher for the same pressure compared to when the pressure decreases. This can be attributed to numerical inaccuracies. In the bottom figure, one can also see the accumulation of samples around 264.5\,bar and 0.8\,kg/min for times after around $t_k=350$\,min, which is when the systems starts to settle around the equilibrium. The curve fit changes little once the maximum influx has been reached, and conditions (\ref{condition samples 1})-(\ref{condition samples 2}) for the inclusion of new samples are not satisfied any more after $t_k=390$\,min.  Compared to the actual influx as given by (\ref{inflow G}), the parameter estimation scheme tends to estimate that the same amount of gas influx occurs at a slightly lower bottom hole pressure (i.e., the estimated samples and fitted curve lie to the left of the black lines in Figure \ref{fig:estimate bottom BC uncertain}), which leads to the offset of approximately 0.5\,bar between the adaptive simulations and the nominal simulation in Figure \ref{fig:pressures uncertain}.

\subsection{\red Monte Carlo Simulations}  \label{sec:MonteCarlo}
{\red
In this section we demonstrate the controller performance in Monte Carlo type simulations with parametric uncertainty and disturbances/noise affecting the measurement and actuation signals. Here, the estimation and control schemes use the nominal parameters given in Table \ref{table parameters}, whereas the actual parameters $f$ (friction factor) and $c_G$ (gas compressibility) used in the drift flux model vary by 5\% around the nominal value. Moreover, we add a random 5\% disturbance/noise signal to the topside measurement of the gas concentration, $\alpha_G(L,t)$, and another random, unmeasured disturbance to the topside pressure $p^{\text{top}}$ of $\pm0.5$\,bar (about 5\% of 10\,bar which is the nominal topside pressure). In order to vary the initial condition, the topside pressure is manually set to in between 8\,bar and 10\,bar, which brings the well into the under-balanced range for most parameter samples and induces a gas influx before the feedback controller is activated after 2\,hours.  Moreover, we show the same simulation for a shorter, 1000\,m deep well. Here, the simulations are run for 100  samples in the given range, including the  18 extreme points where the uncertainty is either 0 or $\pm5$\% and for the initial topside pressure is either 8\,bar or 10\,bar, as well as 82 random samples within this range. The reference pressures are lowered compared to the previous simulations so that despite the error in the friction factor all simulations are in the under-balanced range.

\begin{figure}[htbp!]\centering
\includegraphics[width=.9\columnwidth]{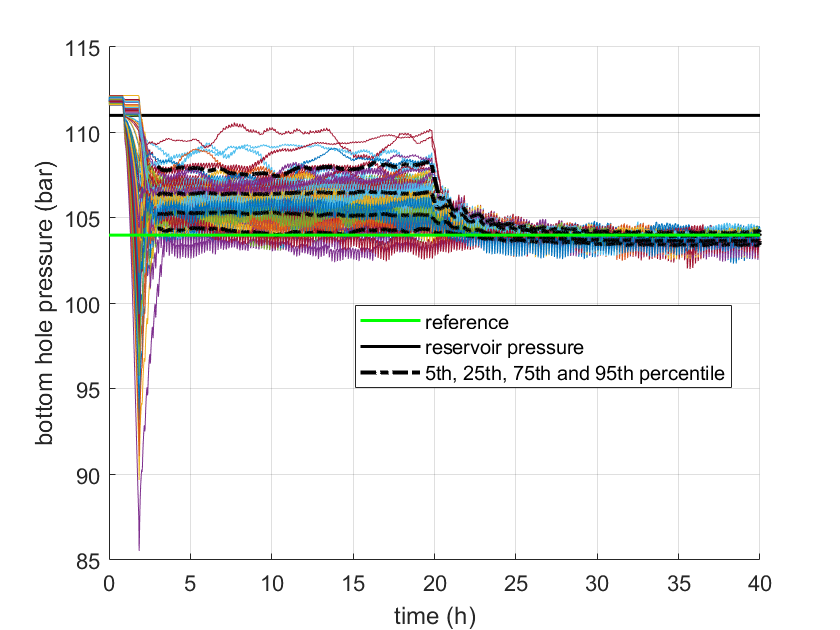}\\
\includegraphics[width=.9\columnwidth]{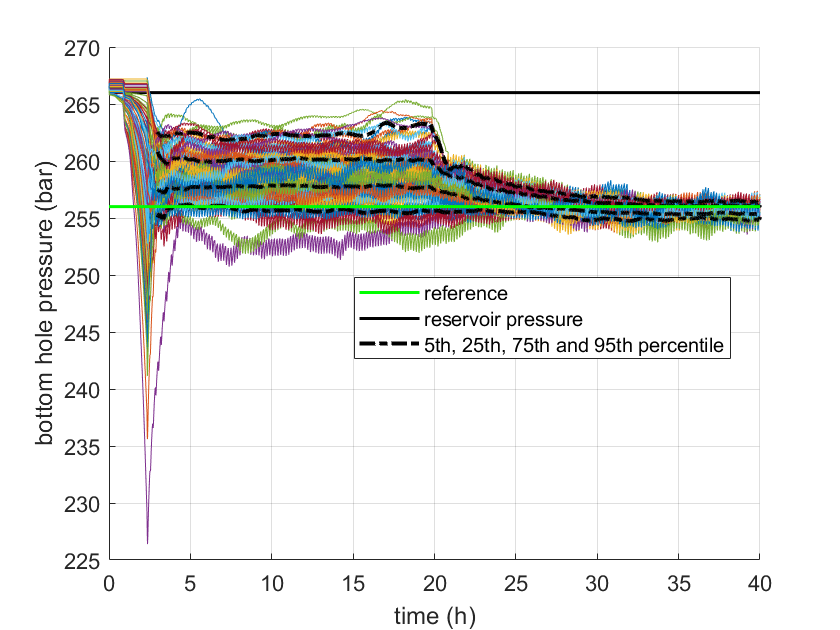}
\caption{\red Monte Carlo simulations with parametric uncertainty and disturbance/noise terms as described in the text, for the 1000\,m (top) and 2500\,m (bottom) deep wells. The plots show the individual bottom hole pressure trajectories (thinner lines) as well as the 5th, 25th, 75th and 95th percentiles. The integral term is activated at $t=20$\,hours.   }
\label{fig:MonteCarlo}
\end{figure}

\red
The simulated trajectories are shown in Figure \ref{fig:MonteCarlo}. The model errors lead to a larger offset between the achieved bottom hole pressures and reference, but the bottom hole pressures stabilize within a few bar of the reference in all simulations. The time-varying noise and disturbance terms cause some fluctuations in the pressure trajectories. 

 In order to compensate the steady-state tracking error, the integral term introduced in Section \ref{sec:integral term} is activated at time $t=20$\,hours with a gain of $K_I = \frac{0.1}{3600\,s}$ and sampling period of 1\,hour. That is, sampling of the downhole pressure is asynchronous with the topside measurement and the integral term gets updated much less frequently than the topside pressure. Still, the integral term helps to quickly bring the bottom hole pressure close to the reference, with minor remaining fluctuations due to the noise and disturbances.

}

\section{Conclusions}
We presented a feedback control design for underbalanced drilling using only measurements and actuators located topside on the drilling rig and with uncertainty in reservoir parameters. In simulations with an industry standard drift-flux model as the plant, the proposed controller manages to stabilize the downhole pressure at an open-loop unstable setpoint slightly below the reservoir pressure (the point considered most difficult to control \cite{Graham2004,Boyun2002}). The method also shows robustness to sampling{\red, modelling errors and disturbances/noise affecting the topside measurements and actuation. For scenarios where infrequent measurements of the bottom hole pressure are available (in the order of once per hour), such measurements can be fed back in an integral fashion to compensate tracking errors caused by parametric uncertainty.  } Independently of whether the presented feedback controller or an alternative strategy is used for pressure control, the proposed estimation scheme provides estimates of the distributed gas concentration, downhole pressure, reservoir pressure, and production index using only topside measurements. Finally, the results serve as a verification that the simplified model from \cite{aarsnes2016simplified} captures the dominant dynamics of the two-phase drift flux mode that are most relevant for control design.

{\red The simulations presented in this paper deviate to some extend from the theoretical stability analysis, in that the model used for control design is different from the plant model. Therefore, in future work the theory should be extended to close this gap. While \cite{strecker2021quasilinear2x2} provides a conservative  robustness analysis for a related system, sharper certificates for robustness with respect to model uncertainty and sampling would be highly desirable. Another direction for  future work would be event-triggered schemes \cite{heemels2012introduction,espitia2020observer}, which might help to further reduce the control effort by only updating the actuation when it is truly needed.
 }

\appendices

\section{\red Proof of well-posedness and convergence  for the simplified drift flux model} \label{appendix}
{\red
In this section, we prove well-posedness of the observer  and control law from Sections \ref{sec:observer} and \ref{subsec:controller}, respectively, as well as stability of the closed-loop system consisting of the simplified drift-flux model (\ref{alpha dynamics simplified})-(\ref{inflow simplified 2}), the observer (Algorithm \ref{estimation algorithm}) and the feedback control law (Algorithm \ref{control algorithm}). Let
\begin{align}
X(x,t) &=  \left(\begin{matrix}\bar{\alpha}_G(x,t)&\bar{p}(x,t)&\bar{v}_G(x,t) \end{matrix}\right)^T,  \label{X(t)} \\
\tilde{X}(x,t) &=  \left(\begin{matrix}\bar{\alpha}_G(x,t)&\bar{p}(x,t)-\bar{p}_0(x)&\bar{v}_G(x,t)-\bar{v}_{G,0}(x) \end{matrix}\right)^T \\
\tilde{Y}(t) &= \tilde{X}(L,t)  , 
\end{align}
where $\bar{p}_0$ and $\bar{v}_{G,0}$ are the steady state pressures and velocities corresponding to zero gas concentration and $\bar{p}_0(0,t)=p_{\text{ref}}$.  Note that $\frac{\partial}{\partial t} {X}(x,t) = \frac{\partial}{\partial t} \tilde{X}(x,t)$ for all $x\in[0,L]$.

\begin{lemma} \label{thm:1}
Fix $k\in\N$. Assume the measurements $Y(t)$ as defined in (\ref{Y(t)}) is Lipschitz-continuous.  There exist  constants $\delta_1>0$ and $\delta^{\prime}_1>0$ such that if $\|\tilde{Y}(t) \|\leq \delta_1$ and $\|\frac{\partial}{\partial_t}\tilde{Y}(t) \|\leq \delta^{\prime}_1$ for all $t\in[t_k-T,t_k]$ with  $T>0$ such that $\tau(t_k;0)\geq t_k-T$ (with $\tau$ as defined in (\ref{tau})), then the system (\ref{alpha dynamics x-direction})-(\ref{v x-direction}) has a unique Lipschitz-continuous solution on the domain $\A(t_k)$. Moreover, there exists a constant $c_1^{\prime}$ such that $\operatorname*{ess\,sup}_{(x,t)\in\A(t_k)}\|\frac{\partial}{\partial t}X(x,t)\|\leq c_1\,\operatorname*{ess\,sup}_{t\in[t_k-T,t_k]} \|\frac{\partial}{\partial t}Y(t)\|$.
\end{lemma}
\begin{proof}
The proof follows the proofs of \cite[Theorem 5]{strecker2021quasilinear2x2} and \cite[Theorem 3.8]{bressan2000hyperbolic}. 
 In order to define broad solutions (see \cite{strecker2021quasilinear2x2}[Theorem 3.8]), we can transform (\ref{alpha dynamics x-direction})-(\ref{v x-direction})  into integral equations by integrating (\ref{alpha dynamics x-direction}) along its characteristic lines. Then, by subtracting the steady state values $\bar{p}_0$ and $\bar{v}_{G,0}$, bounding the integrands by expressions that are locally Lipshitz in the state, exploiting that $\bar{p}$ in the denominator in (\ref{E bar}) is bounded from below by  $\bar{p}(x,t)\geq \bar{p}(L,t) \geq 1$\,bar (so that $\frac{1}{\bar{p}}$ remains bounded), and  using a  Gronwall-type inequality, an a-priori bound on $\|\tilde{X}(x,t)\|$ for $(x,t)\in\A(t_k)$ can be derived. Similarly, integral equations for $\frac{\partial}{\partial t} \tilde{X}(x,t)$ can be derived, the right-hand side of which are super-linear in $\tilde{X}(t)$ and $\frac{\partial}{\partial}\tilde{X}(t)$. Using techniques as in  \cite[Theorem 5]{strecker2021quasilinear2x2}, one can show that the solution of these integral equations do not blow up for all $(x,t)\in\A(t_k)$ if $\delta^{\prime}_1$ is sufficiently small (depending on the bound on $\|\tilde{X}(x,t)\|$ derived previously). Moreover, the integral equations for $\tilde{X}$ and $\frac{\partial}{\partial t}\tilde{X}$ depend  linearly on $\tilde{Y}(t)$ and $\frac{\partial}{\partial t}\tilde{Y}(t)$, respectively, which can be used to bound $\|\frac{\partial}{\partial t}X(x,t)\| $ via $\|\frac{\partial}{\partial t}Y(t)\|$.  
 
 Finally, uniqueness of the solution can be shown by subtracting (\ref{alpha dynamics x-direction})-(\ref{v x-direction}) for two  solutions  with the same topside measurements from each other. Clearly, the zero-solution solves the resulting set of equations, meaning that the two solutions are equal.
\end{proof}

\begin{remark}
In \cite{strecker2021quasilinear2x2}, rigorous expressions for the bounds equivalent to $\delta_1$ and $\delta^{\prime}_1$ in Lemma \ref{thm:1} are given for a related class of quasilinear hyperbolic systems. However, their derivations are extremely technical and are not repeated here. These bounds are based on worst-case growth estimates that are very conservative, meaning that at this stage they are unlikely to give a realistic  estimate that would be of practical value. Moreover, the state $X$ in (\ref{X(t)}) contains concentrations ($<1$) and pressures ($>10^5$\,Pa). This difference in scale would lead to even more conservatism in any bounds, although this could be addressed by rescaling the state. 
\end{remark}

The same techniques as in Lemma \ref{thm:1} can be used to show well-posedness of the second step in the observer evaluation.
\begin{lemma} \label{thm:2}
Assume $X(x,\tau(t_k;x))$ is Lipshitz-continuous in $x\in[0,L]$. There exist  constants $\delta_2>0$ and $\delta^{\prime}_2>0$ such that if $\sup_{x\in[0,L]}\|\tilde{X}(x,\tau(t_k;x)) \|\leq \delta_2$ and $\operatorname*{ess\,sup}_{x\in[0,L]}\|\frac{\partial}{\partial_t}\tilde{X}(x,\tau(t_k;x)) \|\leq \delta^{\prime}_2$, then the system (\ref{alpha dynamics step2})-(\ref{v step2}) has a unique solution on the domain $\B(t_k)$. 
\end{lemma}

Lemmas \ref{thm:1} and \ref{thm:2} form the basis for showing well-posedness and convergence of the observer defined in Algorithm \ref{estimation algorithm}.
\begin{theorem}\label{thm:3}
Assume the measurements $Y(t)$  is Lipschitz-continuous for all $t$.  There exist $K\in\N$ and  constants $\delta_3>0$ and $\delta^{\prime}_3>0$ such that if $\|\tilde{Y}(t) \|\leq \delta_3$ and $\|\frac{\partial}{\partial_t}\tilde{Y}(t) \|\leq \delta^{\prime}_3$ for all $t\in[0,\infty)$, then the state estimate obtained by Algorithm \ref{estimation algorithm} is equal to the actual state at all times $t_k$ with $k\geq K$.
\end{theorem}
\begin{proof}
One can choose $K$ large enough such that $\tau(t_K;0) > 0$. Since both the actual and the estimated trajectory satisfy (\ref{alpha dynamics x-direction})-(\ref{v x-direction}) and are equal to $Y$ at $x=L$, uniqueness of the solution on $\A(t_k)$ for $k\geq K$ as guaranteed  by Lemma \ref{thm:1}, implies that the estimated state is equal to the actual on all of $\A(t_k)$, including on the characteristic line $(x,\tau(t_k;x)$, $x\in[0,L]$. Similarly, since the observer equations (\ref{alpha dynamics step2})-(\ref{v step2}) are just a copy of the set of equations that the actual dynamics satisfy, Lemma \ref{thm:2} implies  that the estimated and actual state exist and are equal on $\B(t_k)$ if $\|\tilde{X}(x,t)\|$ and $\|\frac{\partial}{\partial t}\|\tilde{X}(x,t)\|$ are sufficiently small on the line $(x,\tau(t_k;x)$, $x\in[0,L]$. By the last statement in  Lemma \ref{thm:1}, the latter can be ensured by choosing $\delta_3$ and $\delta^{\prime}_3$ sufficiently small.  Since $\B(t_k)$ includes the line $(x,t_k)$, $x\in[0,L]$, this implies that the estimate of $X(\cdot,t_k)$ is equal to the actual value.
\end{proof}
We next formulate a lemma regarding well-posedness of each feedback control step as given by Algorithm \ref{control algorithm}.
\begin{lemma}\label{thm:4}
Fix $k\in\N$ and assume the state at time $t_k$ is fully known, i.e., (\ref{target IC}) is satisfied. There exist $\delta_4>0$ and $\delta^{\prime}_4>0$ and $\bar{\delta}>0$ such that if $\|\tilde{X}(\cdot,t_k)\|_{\infty} \leq \delta_4$, $\|\frac{\partial}{\partial t}\tilde{X}(\cdot,t_k)\|_{\infty} \leq \delta_4^{\prime}$ and $\bar{p}^{\prime}\leq \bar{\delta}$, then  (\ref{target alpha})-(\ref{target IC}) has a unique solution on $(x,t)\in[0,L]\times[t_k,t_{k+1}]$. Moreover, the actual system (\ref{alpha dynamics simplified})-(\ref{inflow simplified 2}) in closed loop with $p^{\text{top}}$ as constructed by Algorithm \ref{control algorithm} satisfies $\bar{p}(0,t) = p_{\text{ref}}^{*,k}(t)$ for all $t\in[t_k,t_{k+1}]$.
\end{lemma}
\begin{proof}
Existence and uniqueness of the solution on $(x,t)\in[0,L]\times[t_k,t_{k+1}]$ can be proven using the same techniques as in Lemma \ref{thm:1}, where we again use that (\ref{target alpha})-(\ref{target IC}) is just a reformulated version of the actual dynamics. In particular, uniqueness of the solution includes that $\bar{p}(x,t)  = \bar{p}^{*,k}(x,t)$ on $(x,t)\in[0,L]\times[t_k,t_{k+1}]$. That is, $\bar{p}(0,t) = p_{\text{ref}}^{*,k}(t)$ for all $t\in[t_k,t_{k+1}]$ if and only if $\bar{p}(L,t) = p^{\text{top}}(t) =  \bar{p}^{*,k}(L,t)$ for all $t\in[t_k,t_{k+1}]$.
\end{proof}
We are now in position to prove the main theorem on well-posedness and convergence of the closed loop system.
\begin{theorem}\label{thm:5}
Assume the feedback controller is activated at some time $T$ with $\tau(T;0)\geq 0$. There exist  $\delta_5>0$,  $\delta_5^{\prime}>0$, $\tilde{\delta}_5>0$,  $\tilde{\delta}_5^{\prime}>0$, $\bar{\delta}>0$ and $T^{\prime}>T$ such that if the initial conditions and $p^{\text{top}}(t)$ for $t\leq T$ are Lipschitz continuous, compatible and such that the solution exists up to time $T$ with  $\|\tilde{Y}(t) \|\leq \tilde{\delta}_5$ and $\|\frac{\partial}{\partial_t}\tilde{Y}(t) \|\leq \tilde{\delta}^{\prime}_t$ for all $t\leq T$, and such that  $\|\tilde{X}(\cdot,T)\|_{\infty} \leq \delta_5$ and $\|\frac{\delta}{\delta t}\tilde{X}(\cdot,T)\|_{\infty} \leq \delta_5^{\prime}$, then the closed-loop system consisting of the simplified drift flux model (\ref{alpha dynamics simplified})-(\ref{inflow simplified 2}), the observer in Algorithm \ref{estimation algorithm} and the feedback control law in Algorithm \ref{control algorithm} with $\bar{p}^{\prime}\leq \bar{\delta}$ has a unique solution on $[0,L]\times[0,\infty)$ that satisfies $\bar{p}(0,t)=p_{\text{ref}}$ for all $t\geq T^{\prime}$.
\end{theorem}
\begin{proof}
Since $\tau(T;0)\geq 0$ by assumption, Theorem \ref{thm:3} states that the observer has converged by the time the feedback controller is activated. Here, it is assumed that the initial conditions and $p_{\text{top}}(t)$ for $t\leq T$ are benign such that the system is actually observable and controllable by the time the controller is activated.  For all $k$ with $t_k\geq T$, by Lemma \ref{thm:4} the solution satisfies $\bar{p}(0,t) = p_{\text{ref}}^{*,k}(t)$ for $t\in[t_k,t_{k+1}]$. By recursively  using the design in Equation (\ref{pref*}), this means that $\bar{p}(0,t) = p_{\text{ref}}^{*,k}(t) = p_{\text{ref}}$ for all $t\geq T^{\prime} = \frac{|\bar{p}(0,T)-p_{\text{ref}}|}{\bar{p}^{\prime}}$. With regards to well-posedness, the design in (\ref{pref*}) and the assumption that $\|\tilde{X}(\cdot,T)\|_{\infty} \leq \delta_5$, ensures that $\bar{p}(0,t) = p_{\text{ref}}^{*,k}$ remains below a bound that can be made arbitrarily small by making $\delta_5$ small. The norm of the time derivative at $x=0$, $\|\frac{\partial}{\partial t} \tilde{X}(0,t)\|$ for $t\geq T$ can be made arbitrarily small via $\bar{p}^{\prime}$. Then, similar as in Lemma \ref{thm:1}, by solving the dynamics in the positive $x$-direction with the ``initial'' condition at $x=0$, one can show that this implies that $\|\tilde{X}(x,t)\|$ and $\|\frac{\partial}{\partial t} \tilde{X}(x,t)\|$ remain sufficiently small for all $x\in[0,L]$, $t\geq T$. That is, the solution cannot blow up in finite time, and the assumptions of Theorem \ref{thm:3} (smallness of  $\|\tilde{Y}(t) \|$ and $\|\frac{\partial}{\partial_t}\tilde{Y}(t) \|$) and Lemma \ref{thm:4} (smallness of $\|\tilde{X}(\cdot,t_k)\|_{\infty}$ and $\|\frac{\partial}{\partial t}\tilde{X}(\cdot,t_k)\|_{\infty}$) are recursively satisfied. Moreover, the design (\ref{pref*})  is such that $p_{\text{ref}}^{*,k}(t_k)=\bar{p}(0,t_k)$ which, due to (\ref{pbar start L})/(\ref{pbar start 0}) and (\ref{target p}), implies that $p_{\text{top}}(t)$ stays  continuous at $t=t_k$ for all $k$ so that the whole solution remains Lipschitz-continuous.
\end{proof}
\begin{remark}
In \cite{strecker2021quasilinear2x2}, rigorous, although quite conservative certificates for robustness with respect to uncertainty in parameters and measurement and actuation inaccuracies are given for a related class of quasilinear hyperbolic systems. Deriving similar conditions for the system considered here would go beyond the scope of this paper. However, the numerical simulations in Section \ref{sec:MonteCarlo} suggest that there is some inherent robustness with respect to such uncertainties, as well as with respect to mismatch between the full drift-flux model and the siomplified model used for control design. 

The sampling period $\theta$ does not appear in Theorem \ref{thm:5} because in the appendix, exact model knowledge and predictability are assumed. In presence of model uncertainty, the sensitivity of closed-loop stability with respect to $\theta$ is also investigated in \cite{strecker2021quasilinear2x2}. In particular, long $\theta$ can reduce the robustness with respect to model uncertainty due to prediction errors, while very short $\theta$ can also be detrimental because new measurement errors are introduced at every sampling event. The latter can be  managed by introducing  a minimum dwell time (see also the classical reference \cite{hespanha1999dwelltime}).
\end{remark}

}

\bibliographystyle{IEEEtran}
\bibliography{IEEEabrv,references} 

\end{document}